\documentclass[a4paper,10pt]{article}

\usepackage{amsthm}
\usepackage{amsmath}
\usepackage{amssymb}
\usepackage{graphicx}
\usepackage{dsfont}
\usepackage[bf]{caption}

\newtheorem{thm}{Theorem}[section]
\newtheorem{cor}[thm]{Corollary}
\newtheorem{lem}[thm]{Lemma}

\theoremstyle{definition}
\newtheorem{defi}[thm]{Definition}

\newtheoremstyle{rmk}
  {12pt}                   
  {12pt}                   
  {}                       
  {}                       
  {\normalfont\bfseries}   
  {.}                      
  {\newline}               
  {}

\theoremstyle{rmk}
\newtheorem{rmks}[thm]{Remarks}
\newtheorem{rmk}[thm]{Remark}

\renewcommand{\qedsymbol}{$\blacksquare$}

\newcommand{\sN} { \mathbb{N} }
\newcommand{\sZ} { \mathbb{Z} }
\newcommand{\sR} { \mathbb{R} }
\newcommand{\sQ} { \mathbb{Q} }

\newcommand{\esssup}{\mathop{\mathrm{ess\,sup}}\displaylimits}
\newcommand{\bea}{\begin{eqnarray*}}
\newcommand{\eea}{\end{eqnarray*}}

\begin{document}\parindent 0pt
\title{Survival and Growth of a Branching Random Walk in Random Environment}
\author{Christian Bartsch, Nina Gantert and Michael Kochler}
\date{\today}
\maketitle
We consider a particular Branching Random Walk in Random Environment (BRWRE) on $\sN_0$ started with one particle at the origin. Particles reproduce according to an offspring distribution (which depends on the location) and move either one step to the right (with a probability in $(0,1]$ which also depends on the location) or stay in the same place. We give criteria for local and global survival and show that global survival is equivalent to exponential growth of the moments. Further, on the event of survival the number of particles grows almost surely exponentially fast with the same growth rate as the moments.\\

\section{Introduction}
\bigskip
We consider a particular Branching Random Walk in Random Environment (BRWRE) on $\sN_0$ started with 
one particle at the origin.
The environment is an i.i.d.\ collection of offspring distributions 
and transition probabilities. In our model particles can either move one step to the right or they can stay where they are. Given a realization of the environment, we consider a random cloud of particles evolving as follows.
We start the process with one particle at the origin, and then repeat the following two steps:
\begin{itemize}
\item
Each particle produces offspring independently of the other particles according to the offspring distribution at its location (and then it dies).
\item
Then all particles move independently of each other.
Each particle either moves to the right (with probability $h_x$, where $x$ is the location of the particle) or it stays at its position (with probability  $1-h_x$).
\end{itemize}
We are interested in survival and extinction of the BRWRE and in the connection between survival/extinction and the (expected) growth rate of the number of particles. Further, we  characterize the profile of the expected number of particles on $\sN_0$.
The question on survival/extinction is considered for particles moving to the left or to the right in a paper by Gantert, M\"{u}ller, Popov and Vachkovskaia, see \cite{gantert}. Our model is excluded by the assumptions in \cite{gantert} (Condition~E). The questions on the growth rates are motivated by a series of papers by Baillon, Clement, Greven and den Hollander, see \cite{greven1}, \cite{greven5}, \cite{greven2}, \cite{greven3} and \cite{greven4}, where a similar model starting with one particle at each location is investigated. Since in such a model the global population size is always infinite, the authors introduce different quantities to describe the local and global behaviour of the system. They apply a variational approach to analyse different
growth rates. We give a different (and easier) characterization of
the global survival regime, using an embedded branching process in random environment. For the connection between this paper and the model in \cite{greven2} see Remark~\ref{connection}.\\
To get results for the growth of the global population (Theorem~\ref{existenz des limes} and Theorem~\ref{thm GS und Zn}) it is useful to investigate the local behaviour of the process which is done with the help of the function $\beta$ in Theorem~\ref{prop local growth}. The function $\beta$ describes the profile of the expected number of particles. However, $\beta$ is not very explicit: its existence follows from the subadditive ergodic theorem. In the proofs of these theorems we follow the ideas of a paper by Comets and Popov \cite{popov}. An important difference to \cite{popov} is that in our model particles can have no offspring. To determine the growth rate of the population, we have to condition on the event of survival.\\
If $h\equiv 1$, the spatial component is trivial (in this case, all particles at time $n$ are located at $n$) and the model reduces to the well-known branching process in random enviroment, see \cite{tanny}. Our results can be interpreted as extensions of the results in \cite{tanny} for processes in time and space.

The paper is organized as follows. In Section 2 we give a formal description of our model. Section 3 contains the results, Section 4 some remarks and Section~5 the proofs. At last, in Section 6 we provide examples and pictures.

\section{Formal Description of the Model}
The considered BRWRE will be constructed in two steps, namely we first choose an environment and then let the particles reproduce and move in this environment.
\vspace{12pt}\\
\textbf{Step I}\ (Choice of the environment)\vspace{2mm}\\
First, define
$${\cal M}:=\Big\{ (p_i)_{i \in\sN_0}:p_i\geq0, \sum_{i=0}^\infty p_i=1 \Big\}$$
as the set of all offspring distributions (i.e.\ probability measures on $\sN_0$).
Then, define
$$\Omega:= {\cal M} \times (0,1]$$
as the set of all possible choices for the local environment, now also containing the local drift parameter. Let $\alpha$ be a probability measure on $\Omega$ satisfying\renewcommand{\arraystretch}{2}\begin{equation}\label{ellipticity}
 \begin{array}{c}
  \alpha\left( \Big\{ \big((p_i)_{i\in\sN_0}, h \big) \in\Omega:  p_1=1  \Big\} \right)<1,\\
  \alpha\left( \Big\{ \big((p_i)_{i\in\sN_0}, h \big) \in\Omega: p_0\leq1-\delta, \, h \in [\delta,1] \Big\} \right)=1
 \end{array}
 \end{equation}\renewcommand{\arraystretch}{1}for some $\delta > 0$. The first property ensures that the branching is non-trivial and the second property is a common ellipticity condition which comes up in the context of survival of branching processes in random environment.\\
Let $\omega=(\omega_x)_{x\in\sN_0}=(\mu_x,h_x)_{x\in\sN_0}$ be an i.i.d.\ random sequence in $\Omega$ with distribution $\alpha^{\sN_0}=\bigotimes_{x\in\sN_0}\alpha$. We write $\textsf{P}:=\alpha^{\sN_0}$ and $\textsf{E}$ for the associated expectation. In the following $\omega$ is referred to as the random environment containing the offspring distributions $\mu_x$ and the drift parameters $h_x$. Let
$$m_x=m_x(\omega):=\sum_{k=0}^\infty k\mu_x\big(\{k\}\big)$$
be the mean offspring at location $x\in\sN_0$. We denote the essential supremum of $m_0$ by
$$M:=\esssup m_0$$
and furthermore we define
$$\Lambda:=\esssup\big(m_0 (1-h_0)\big).$$\vspace{12pt}\\
\textbf{Step II}\ (Evolution of the cloud of particles)\vspace{2mm}\\
Given the randomly chosen environment $(\omega_x)_{x \in \sN_0}=(\mu_x,h_x)_{x \in \sN_0}$, the cloud of particles evolves at every time $n \in \sN_0$. First each existing particle at some site $x \in \sN_0$ produces offspring according to the distribution $\mu_x$ independently of all other particles and dies. Then the newly produced particles move independently according to an underlying Markov chain starting at position $x$. The transition probabilities are also given by the environment. We will only consider a particular type of Markov chain on $\sN_0$ that we may call ``movement to the right with (random) delay''.
This Markov chain is determined by the following transition probabilities:
\begin{equation}
p_\omega(x,y)=\begin{cases}h_x& y=x+1\\1-h_x& y=x\\0& \text{otherwise}\end{cases}
\end{equation}
Note that due to the ellipticity condition \eqref{ellipticity}, $h_x$ is bounded away from 0 by some positive $\delta$.
Later, we consider the case that \textsf{P}$(h_0=h)=1$ for some $h\in(0,1]$ where the drift parameter is constant and analyse different survival regimes depending on the drift parameter $h$, see Theorem \ref{phase transitions}.\\
For $n\in\sN_0$ and $x\in\sN_0$ let us denote the number of particles at location $x$ at time $n$ by $\eta_n(x)$ and furthermore let
$$Z_n:=\sum_{x\in\sN_0}\eta_n(x)$$
be the total number of particles at time $n$.\\
We denote the probability and the expectation for the process in the fixed environment $\omega$ started with one particle at $x$ by $P_\omega^x$ and $E_\omega^x$, respectively. $P_\omega^x$ and $E_\omega^x$ are often referred to as ``quenched'' probability and expectation.\\
\newpage
Now we define two survival regimes:
\begin{defi}\label{survival}
Given $\omega$, we say that
\begin{enumerate}
 \item
there is \textit{Global Survival} (GS) if
$$ P_\omega^0\big(Z_n\to0\big)<1.$$
 \item
there is \textit{Local Survival} (LS) if
$$ P_\omega^0\big(\eta_n(x)\to0\big)<1$$
for some $x\in\sN_0$.
\end{enumerate}
\end{defi}
\begin{rmks}
\begin{enumerate}
\item
For fixed $\omega$ LS is equivalent to
$$P_\omega^0\big(\eta_n(x)\to0\ \forall\; x\in\sN_0\big)<1.$$
\item
Since the drift parameter is always positive, it is easy to see that for fixed $\omega$ LS and GS do not depend on the starting point in Definition \ref{survival}. Thus we will always assume that our process starts at $0$. For convenience we will omit the superscript 0 and use $P_\omega$ and $E_\omega$ instead.
\end{enumerate}
\end{rmks}
\section{Results}
The following results characterize the different survival regimes. As in \cite{gantert}, local and global survival do not depend on the realization of the environment but only on its law.

\begin{thm}\label{LS}
There is either LS for \textnormal{\textsf{P}}-a.e. $\omega$ or there is no LS for \textnormal{\textsf{P}}-a.e. $\omega$.\\
There is LS for \textnormal{\textsf{P}}-a.e. $\omega$ iff
$$\Lambda>1.$$
\end{thm}

\begin{thm}\label{GS}
Suppose $\Lambda\leq1$.\\
There is either GS for \textnormal{\textsf{P}}-a.e. $\omega$ or there is no GS for \textnormal{\textsf{P}}-a.e. $\omega$.\\
There is GS for $\textnormal{\textsf{P}}$-a.e. $\omega$ iff
$$\textnormal{\textsf{E}}\left[\log\left(\frac{m_0h_0}{1-m_0(1-h_0)}\right)\right]>0.$$
\end{thm}\vspace{12pt}

We now consider the local and the global growth of the moments $E_\omega[\eta_n(x)]$ and $E_\omega[Z_n]$. For Theorems~\ref{prop local growth} -- \ref{thm GS und Zn}, we need the following stronger condition\renewcommand{\arraystretch}{2}\begin{equation}\label{strong ellipticity}
\begin{array}{c}
\alpha\left( \Big\{ \big((p_i)_{i\in\sN_0}, h \big) \in\Omega: p_1=1 \Big\} \right)<1,\\
\alpha\left( \Big\{ \big((p_i)_{i\in\sN_0}, h \big) \in\Omega: p_0\leq1-\delta, \, h \in [\delta,1-\delta] \Big\} \right)=1
\end{array}
\end{equation}\renewcommand{\arraystretch}{1}for some $\delta >0$. In addition, for those theorems we assume $M<\infty$.

\begin{thm}\label{prop local growth}
There exists a unique, deterministic, continuous and concave function $\beta:[0,1]\longrightarrow\sR$ such that for every $\gamma>0$ we have for \textnormal{\textsf{P}}-a.e. $\omega$
$$\lim_{n\to\infty}\ \max_{x\in n[\gamma,1]\cap\sN}\Big| \tfrac{1}{n}\log E_\omega\big[\eta_n(x)\big]-\beta(\tfrac{x}{n}) \Big|=0.$$
Additionally, it holds that $\beta(0)=\log\big(\Lambda\big)$ and $\beta(1)=\textnormal{\textsf{E}}[\log(m_0h_0)]$.
\end{thm}

\begin{thm}\label{existenz des limes}
We have
$$\lim_{n\to\infty}\tfrac{1}{n}\log E_\omega\big[Z_n\big] = \max_{x\in[0,1]}\beta(x)\quad\text{for }\textnormal{\textsf{P}}\text{-a.e. }\omega.$$
\end{thm}\vspace{12pt}

The next theorem shows that GS is equivalent to exponential growth of the moments $ E_\omega[Z_n]$:

\begin{thm}\label{GS und EW}
The following assertions are equivalent:
 \begin{enumerate}
 \item
 $\lim\limits_{n\to\infty} \tfrac{1}{n}\log E_\omega\big[Z_n\big] > 0$ holds for $\textnormal{\textsf{P}}\text{-a.e. }\omega.$
 \item
 There is GS for \textnormal{\textsf{P}}-a.e. $\omega$.
 \end{enumerate}
\end{thm}\vspace{12pt}

In the following theorem we consider the growth of the population $Z_n$ on the event of survival:

\begin{thm}\label{thm GS und Zn}
If there is GS we have for \textnormal{\textsf{P}}-a.e. $\omega$
$$\lim_{n\to\infty}\tfrac{1}{n}\log Z_n =\max_{x\in[0,1]}\beta(x)>0\quad\text{P}_{\omega}\text{-a.s.}\quad\text{on}\quad\{Z_n\not\to 0\}.$$ 
\end{thm}\vspace{12pt}

As already announced above we now analyse the case of constant drift parameter, i.e.\ $\textsf{P}(h_0=h)=1$ for some $h\in(0,1]$. As it is easy to see from Theorem~\ref{LS} in this case we have LS iff
$$h<h_{LS}:= \begin{cases} 1-\frac{1}{M} & \quad \text{if }M\in(1,\infty] \\ 0&\quad\text{if }M\in(0,1]\ \ .\end{cases}$$
To analyse the dependence of GS on $h$ we define
$$\varphi(h):=\textnormal{\textsf{E}}\left[\log\left(\frac{m_0h}{1-m_0(1-h)}\right)\right].$$

\begin{thm}\label{phase transitions}
Suppose $h\geq h_{LS}$.
\begin{enumerate}
\item
If $M\leq1$ then we have $\varphi(h)\leq0$ for all $h\in(0,1]$ and thus there is a.s.\ no GS.
\item
Assume $M>1$.
\begin{enumerate}
\item
If $\varphi(h_{LS})\geq0$ and $\varphi(1)\leq0$ then there is a unique $h_{GS}\in[h_{LS},1]$ with $\varphi(h_{GS})=0$. In this case we have a.s.\ GS for $h\in(0,h_{GS})$ and a.s.\ no GS for $h\in[h_{GS},1]$.
\item
If $\varphi(h_{LS})<0$ then $\varphi(h)<0$ for all $h\in[h_{LS},1]$. Thus we have a.s.\ GS for $h\in(0,h_{LS})$ and a.s.\ no GS for $h\in[h_{LS},1]$. In this case we define $ h_{GS}:= h_{LS}$.
\item
If $\varphi(1)>0$ then $\varphi(h)>0$ for all $h\in[h_{LS},1]$. Thus there is a.s.\ GS for all $h\in(0,1]$. In this case we define $h_{GS}:=\infty$.
\end{enumerate}
Hence, we have a unique $h_{GS}\in[h_{LS},1]\cup\{\infty\}$ such that there is a.s.\ GS for $h< h_{GS}$ and a.s.\ no GS for $h\geq h_{GS}$.
\end{enumerate}
\end{thm}

\section{Remarks}
The following remarks apply to the case of constant drift.
\begin{rmks}\label{rmk}
\begin{enumerate}
\item
Since $\varphi(1)=\textsf{E}[\log m_0]$, our results can be seen as an extension of the well-known condition for possible survival of branching processes in a random environment (see Theorem~5.5 and Corollary~6.3 in \cite{tanny}, recalling that we assume condition~\eqref{ellipticity}). In fact, our proofs rely on this result.
\item
If $M<\infty$ and $\varphi(h_{LS})\in(0,\infty]$ then due to the continuity of $\varphi$ there exists $z>0$ such that
there is a.s.\ GS but a.s.\ no LS for every $h\in[h_{LS},h_{LS}+z)$. In particular, this is the case if $\textsf{P}(m_0=M)>0$, since then $\varphi(h_{LS})=\infty$.
\item
We provide an example for a setting in which the condition of Theorem~\ref{phase transitions}~\textit{(ii)(b)} holds. In this case there is a.s.\ LS for $h\in(0,h_{LS})$ and a.s.\ no GS for $h\in[h_{LS},1]$ for some $h_{LS}\in(0,1)$. (See Section \ref{examples}.)
\end{enumerate}
\end{rmks}

\begin{rmk}\label{connection}
The expected global population size $E_\omega[Z_n]$ corresponds to $d_n^{\text{I}}(0,F)$ in the notation of \cite{greven2}. In Theorem 2 I. they describe the limit
$$\lim_{n\to\infty}\tfrac{1}{n}\log E_\omega[Z_n]=\lim_{n\to\infty}\tfrac{1}{n}\log d_n^{\text{I}}(0,F) =:\lambda(h)$$
as a function of the drift $h$ by an implicit formula.\\
To see this correspondence let $(S_n)_{n\in\sN_0}$ be a random walk with (non-random) transition probabilities $(p_h(x,y))_{x,y\in\sN_0}$ starting in $0$ where the transition probabilities are defined by
$$p_h(x,y):=\begin{cases}h& y=x+1\\1-h& y=x\\0& \text{otherwise}\end{cases}$$
and let $E_h$ be the associated expectation. We denote the local times of $(S_n)_{n\in\sN_0}$ by $l_n(x)$, that is
$$l_n(x):=| \{ 0\leq i\leq n: S_i = x\}|\quad\text{for }x\geq0,n\geq0.$$
For $x=0$ we now have
$$E_\omega[\eta_n(0)]=(1-h)^n\cdot m_0(\omega)^n=E_h\left[\prod_{i=0}^{n-1}m_{S_i}(\omega)\cdot\mathds{1}_{\{S_n=0\}}\right].$$
For $x\geq1$ we have
$$E_\omega[\eta_n(x)]=h\cdot m_{x-1}(\omega)\cdot E_{\omega}\big[\eta_{n-1}(x-1)\big]+(1-h)\cdot m_{x}(\omega)\cdot E_{\omega}\big[\eta_{n-1}(x)\big]$$
which yields
$$E_\omega[\eta_n(x)]=E_h\left[\prod_{i=0}^{n-1}m_{S_i}(\omega)\cdot\mathds{1}_{\{S_n=x\}}\right]$$
for all $x\geq1$ by induction.
Finally we get
$$E_\omega[Z_n]=\sum_{x=0}^\infty E_{\omega}[\eta_n(x)]=E_h\left[ \prod_{i=0}^{n-1} m_{S_i}(\omega) \right]=E_h\left[ \prod_{x=0}^{n-1} m_x(\omega)^{l_n(x)} \right].$$
Since we can extend the environment $\omega=(\omega_x)_{x\in\sN_0}$ to an i.i.d.\ environment $(\omega_x)_{x\in\sZ}$ and since $(\omega_x)_{x \in \sZ}$ and $(\omega_{-x})_{x \in \sZ}$ have the same distribution with respect to \textsf{P}, formula (1.8) and Theorem 1 in \cite{greven2} show that there exists a deterministic $c \in \sR$ such that
$$\lim_{n \to \infty} \tfrac{1}{n} \log E_\omega[Z_n] = c\quad\text{for }\textsf{P}\text{-a.e. }\omega.$$
In our notation this limit coincides with $\max_{x\in[0,1]}\beta(x)$.\\
The connection between the two models enables us to characterize the critical drift parameter at which the function $h\mapsto\lambda(h)$ in \cite{greven2} changes its sign using an easier criterion, see Theorem \ref{phase transitions}.
\end{rmk}
\section{Proofs}
\begin{proof}[\bf Proof of Theorem \ref{LS}]
First we observe that the descendants of a particle at location $x$ that stay at $x$ form a Galton-Watson process with mean offspring $m_x(1-h_x)$. Given $\omega$, we therefore have
$$P_\omega^x\big(\eta_n(x)\to0\big)<1\quad\quad\Leftrightarrow\quad\quad m_x(\omega)(1-h_x(\omega))>1.\vspace{12pt}$$
Now assume $\Lambda>1$. Thus there is some $\lambda > 1$ such that
$$ \textsf{P}(m_0(1-h_0)\geq\lambda)>\varepsilon>0$$
for some $\varepsilon>0$ and using the Borel-Cantelli lemma we obtain that \textsf{P}-a.s.\ for infinitely many locations $x$ we have
$$m_x(1-h_x)>1.$$
Let $x_0=x_0(\omega)$ be a location satisfying $m_{x_0}(1-h_{x_0})>1$.\\
For \textsf{P}-a.e. $\omega$ we see
\bea
&&P_\omega(\eta_{x_0}(x_0)\geq1)\\
&\geq&\big(1-\mu_0\big(\{0\}\big)\big)h_0\cdot\big(1-\mu_1\big(\{0\}\big)\big)h_1\cdot\ldots\cdot\big(1-\mu_{x_0-1}\big(\{0\}\big)\big)h_{x_0-1}\\
&>&0
\eea
whereas the second inequality uses condition \eqref{ellipticity}.\\
We obtain for \textsf{P}-a.e. $\omega$
\bea
&&P_\omega(\eta_n({x_0})\to\infty)\\
&\geq& P_\omega(\eta_{x_0}({x_0})\geq1)\cdot P_\omega^{x_0}(\eta_n({x_0})\to\infty)\\
&>&0
\eea
and thus LS.\vspace{12pt}\\
Now assume $\Lambda\leq1$. As mentioned above, for every $x\in\sN_0$ and \textsf{P}-a.e.\ $\omega$ the descendants of a particle at 
location $x$ that stay at $x$ form a subcritical or critical Galton-Watson process. Thus for a given $\omega$ we have
$$\eta_n(0)\to0\quad P_\omega\text{-a.s.}$$
and the total number of particles that move from $0$ to $1$ is therefore $P_\omega$-a.s. finite. Inductively we conclude for every $x\in\sN_0$ that the total number of particles that reach location $x$ from $x-1$ is finite. By assumption each of those particles starts a subcritical or critical Galton-Watson process at location $x$ which dies out $P_\omega$-a.s.. This implies 
$$P_{\omega}(\eta_n(x)\to0)=1\quad\forall\ x\in\sN_0$$
which completes the proof.
\end{proof}
\begin{proof}[\bf Proof of Theorem \ref{GS}]
Since $\Lambda\leq1$ by assumption, there is \textsf{P}-a.s.\ no LS according to Theorem \ref{LS}. In other words we have for all $x\in\sN_0$
$$P_\omega(\eta_n(x)\to0)=1\quad\text{for }\textsf{P}\text{-a.e. }\omega.$$
We now define a branching process in random environment $(\xi_n)_{n\in\sN_0}$ that is embedded in the considered BRWRE. After starting with one particle at $0$ we freeze all particles that reach $1$ and keep those particles frozen until all existing particles have reached $1$. This will happen a.s.\ after a finite time because the number of particles at $0$ constitutes a subcritical or critical Galton-Watson process that dies out with probability $1$. We now denote the total number of particles frozen in $1$ by $\xi_1$. Then we release all particles, let them reproduce and move according to the BRWRE and freeze all particles that hit $2$. Let $\xi_2$ be the total number of particles frozen at $2$. We repeat this procedure and with $\xi_0:=1$ we obtain the process $(\xi_n)_{n\in\sN_0}$ which is a branching process in an i.i.d. environment.\\
Another way to construct $(\xi_n)_{n\in\sN_0}$ is to think of ancestral lines. Each particle has a unique ancestral line leading back to the first particle starting from the origin. Then, $\xi_k$ is the total number of particles which are the first particles that reach 
position $k$ among the particles in their particular ancestral lines.\\
We observe that 
GS of $(Z_n)_{n\in\sN_0}$ is equivalent to survival of $(\xi_n)_{n\in\sN_0}$.\\
Due to Theorem~5.5 and Corollary~6.3 in \cite{tanny} (taking into account condition \eqref{ellipticity}) the process $(\xi_n)_{n\in\sN_0}$ survives with positive probability for \textsf{P}-a.e.\ environment $\omega$ iff
$$\int\, \log\big( E_\omega[\xi_1]\big)\, \textsf{P}(d\omega)>0.$$
Computing the expectation $E_\omega[\xi_1]$ now completes our proof. First we define $\xi_1^{(k)}$ as the number of particles which move from position $0$ to $1$ at time $k$. Using this notation we may write
$$\xi_1=\sum_{k=0}^\infty\xi_1^{(k)}$$
and obtain
$$E_\omega[\xi_1]=\sum_{k=0}^\infty E_\omega\big[\xi_1^{(k)}\big]\, .$$
To calculate $E_\omega\big[\xi_1^{(k)}\big]$ we observe that (with respect to $P_\omega$) the expected number of particles at 
position $0$ at time $k$ equals $\big(m_0(\omega)\cdot(1-h_0(\omega))\big)^k$. Each of those particles contributes $m_0(\omega)\cdot h_0(\omega)$ to $E_\omega\big[\xi_1^{(k)}\big]$. This yields
\begin{eqnarray}\label{expactation xi}
E_\omega[\xi_1]&=&\sum_{k=0}^\infty \big(m_0(\omega)\cdot(1-h_0(\omega))\big)^k \cdot m_0(\omega)\cdot h_0(\omega)\nonumber\\
               &=&\frac{m_0(\omega)\cdot h_0(\omega)}{1-m_0(\omega)\cdot(1-h_0(\omega))}
\end{eqnarray}
which is defined as $\infty$ if $m_0(\omega)\cdot(1-h_0(\omega))=1$.
\end{proof}
\begin{rmk}
Alternatively, equation \eqref{expactation xi} can be obtained using generating functions. The crucial observation is that the generating function $f_x(s):=E_\omega[s^{\xi_{x+1}} | \xi_x=1]$ is a solution of the equation
$$f_x(s)=g_x\big((1-h_x)f_x(s)+h_xs\big)$$
where $g_x(s):=\sum_{k=0}^\infty \mu_x\big(\{k\}\big)s^k$. Then, $E_\omega[\xi_1] = f_0^\prime(1)$, leading to \eqref{expactation xi} .
\end{rmk}

\begin{proof}[\bf Proof of Theorem \ref{prop local growth}]
Following the ideas of \cite{popov} we introduce the function $\beta$ to investigate the local growth rates.\\
\newline
(i) First we show that $\beta$ can be defined as a concave function on $(0,1]\cap \sQ$ such that
\begin{equation}\label{beta on Q}
\lim_{n\to\infty} \tfrac{1}{sn} \log E_\omega\big[\eta_{sn}(rn)\big]=\beta\left(\tfrac{r}{s}\right)
\end{equation}
holds for all $r,s\in\sN$ with $r\leq s$ and for \textsf{P}-a.e.\ $\omega$.\\
To see this fix $r,s\in\sN$ with $r\leq s$. We define
$$S_{m,n}(\omega):=\tfrac{1}{s} \log E_\omega^{rm}\big[\eta_{s(n-m)}(rn)\big]$$
for $0\leq m\leq n$ which is integrable due to \eqref{strong ellipticity} and $M<\infty$. Using this definition, we have
\begin{equation}\label{shift invarianz}
S_{m+1,n+1}(\omega)=S_{m,n}\circ \Theta(\omega)
\end{equation}
where $\Theta(\omega):=\theta^r(\omega)$ with $\theta$ denoting the shift operator as usual, i.e.\ $(\theta\,\omega)_i = \omega_{i+1}$. Furthermore we have
\begin{equation}\label{superadd}
S_{0,n}(\omega)\geq S_{0,m}(\omega)+S_{m,n}(\omega)
\end{equation}
since
\begin{equation*}
E_\omega^0\big[ \eta_{sn}( rn )\big]\geq E_\omega^{0}\big[ \eta_{sm}( rm )\big]\cdot E_\omega^{rm}\big[ \eta_{s(n-m)}( rn )\big].
\end{equation*}
With the properties \eqref{shift invarianz} and \eqref{superadd} we are able to apply the subadditive ergodic theorem to $(S_{m,n})$ and we obtain that
$$\lim_{n\to\infty} \tfrac{1}{n} S_{0,n}(\omega)=\lim_{n\to\infty} \tfrac{1}{sn} \log E_\omega\big[\eta_{sn}(rn)\big]=:\beta\left(\tfrac{r}{s}\right)$$
exists for \textsf{P}-a.e.\ $\omega$. Clearly, the limit only depends on $\tfrac{r}{s}$. Whereas it is \textsf{P}-a.s.\ constant since \textsf{P} is i.i.d..\\
\newline
(ii) We now show that $\beta$ is concave on $(0,1]\cap \sQ$. Fix $a,b,t\in(0,1]\cap\sQ$ with $t\neq1$ and let $s:=a'\!\cdot b'\!\cdot t'$ be the product of the denominators of the reduced fractions of $a,b,t$. Due to \eqref{superadd} we have
\begin{eqnarray}\label{concav}
&&\tfrac{1}{sn}\log E_\omega\Big[ \eta_{sn} \big( s( ta+(1-t)b)n\big) \Big]\nonumber\\
&\geq& t\tfrac{1}{stn}\log E_\omega\Big[ \eta_{stn} \big( stan\big) \Big]\nonumber\\
&&\quad +\ (1-t)\tfrac{1}{s(1-t)n}\log E_\omega^{stan} \Big[ \eta_{s(1-t)n}\big(  s( ta+(1-t)b)n \big) \Big]\nonumber\\
&=& t\tfrac{1}{stn}\log E_\omega\Big[ \eta_{stn} \big( stan\big) \Big]\nonumber\\
&&\quad +\ (1-t)\tfrac{1}{s(1-t)n}\log E_{\theta^{stan}\omega}\Big[ \eta_{s(1-t)n} \big(  s(1-t)bn \big) \Big].
\end{eqnarray}
We observe that for all $n\in\sN_0$
$$E_{\theta^{stan}\omega}\Big[ \eta_{s(1-t)n} \big(  s(1-t)bn \big) \Big]\overset{d}{=}E_\omega\Big[ \eta_{s(1-t)n} \big(  s(1-t)bn \big) \Big].$$
Due to \eqref{beta on Q} and since $\beta$ is \textsf{P}-a.s.\ constant, this implies
$$(1-t)\tfrac{1}{s(1-t)n}\log E_{\theta^{stan}\omega}\Big[ \eta_{s(1-t)n} \big(  s(1-t)bn \big) 
\Big]\xrightarrow[n\to\infty]
{\textnormal{\textsf{ }}}
(1-t)\beta(b)$$
in probability. Therefore there exists a subsequence such that we have \textsf{P}-a.s.\ convergence in \eqref{concav} and this yields
$$\beta( ta+(1-t)b)\geq t\beta(a)+(1-t)\beta(b).$$
We observe that $\beta$ is bounded with $2\log\delta+\log(1-\delta)\leq\beta(x)\leq\log M$ and thus it can be uniquely extended to a continuous and concave function $\beta:(0,1)\longrightarrow\sR$.\\
\newline
(iii) We now investigate the behaviour of $\beta$ for $x \downarrow 0$ and show that 
$$\lim_{x \downarrow 0} \beta(x) = \log(\Lambda).$$
Fix $\varepsilon>0$ and $a \in \mathbb{Q} \cap (0,\varepsilon]$. Let $a'$ be the denominator of the reduced fraction of $a$. For \textnormal{\textsf{P}}-a.e. $\omega$ there exists $y=y(\omega)$ with
$$m_{y(\omega)} (1-h_{y(\omega)}) > \Lambda - \varepsilon.$$
With
$$k:= \max \{ l \in \sN: l \leq (1-\varepsilon)a'n \}$$
we get for large $n$ such that $k \geq y(\omega)$ 
\bea
E_{\omega}\big[\eta_{a'n}(a'an)\big]&\geq&E_{\omega}\big[\eta_{k}(y(\omega))\big] \cdot E^{y(\omega)}_{\omega}\big[\eta_{a'n-k}(a'an)\big] \\
&\geq&\delta_0^{y(\omega)} \cdot (\Lambda - \varepsilon)^{k-y(\omega)} \cdot \delta_0^{a'n - k} \quad \text{for }\textsf{P}\text{-a.e. }\omega
\eea
whereas $\delta_0:=\delta^2\cdot(1-\delta)$. Taking $n \to \infty$ and $\varepsilon \to 0$ we conclude
$$\liminf_{x \downarrow 0} \beta(x) \geq \log(\Lambda).$$
To get the other inequality we notice that for $n_1,n_2\in\sN$ we have
\begin{equation}\label{ineq1}
E_{\omega} \big[\eta_{n_1 \cdot n_2}(n_2)\big] \leq \tbinom{n_1 \cdot n_2}{n_2} \cdot \Lambda^{(n_1-1)\cdot n_2} \cdot M^{n_2} \quad \text{for }\textsf{P}\text{-a.e. }\omega. 
\end{equation}
Since
$$ \tfrac{1}{n_1 \cdot n_2} \log \tbinom{n_1 \cdot n_2 }{ n_2} \xrightarrow[n_2 \to \infty]{} \tfrac{n_1-1}{n_1} \log \big( \tfrac{n_1}{n_1-1} \big) + \tfrac{1}{n_1} \log(n_1) \xrightarrow[n_1 \to \infty]{} 0,$$
\eqref{ineq1} yields for \textsf{P}-a.e.\ $\omega$
\bea
\tfrac{1}{n_1 \cdot n_2} \log E_{\omega} \big[ \eta_{n_1 \cdot n_2} (n_2)\big] &\leq& (o(n_2)+o(n_1))+ \tfrac{n_1-1}{n_1} \log(\Lambda) + \tfrac{1}{n_1} \log(M)\\
&& \xrightarrow[n_2 \to \infty]{} \tfrac{n_1-1}{n_1} \log(\Lambda) + o(n_1).
\eea
This implies
$$\limsup_{n \to \infty} \beta\left(\tfrac{1}{n} \right) \leq \log(\Lambda)  $$
and due to the continuity of $\beta$ on $(0,1)$ we conclude
$$ \limsup_{x \downarrow 0} \beta(x) \leq \log(\Lambda).$$
\newline
(iv) Since $(\eta_n(n))_{n\in\sN_0}$ is a branching process in an i.i.d.\ environment satisfying $E_{\omega}[\eta_1(1)]=~m_0 h_0$, we have $$\beta(1)=\textnormal{\textsf{E}}\big[\log(m_0h_0)\big].$$
The continuity of $\beta$ in 1 can be shown with similar arguments as in part (iii).
\\
\newline
(v) Fix $\gamma>0$ and $\varepsilon>0$. We now show that for \textsf{P}-a.e.\ $\omega$
\begin{equation}\label{prop2}
\liminf_{n\to\infty} \min_{x\in n[\gamma,1]\cap\sN}\Big( \tfrac{1}{n}\log E_\omega[\eta_n(x)]-\beta(\tfrac{x}{n})\Big) \geq0.
\end{equation}
To see this we observe that there is a finite set $\{a_1,\ldots,a_l\}\subset(0,1)\cap\sQ$ satisfying the following condition:
$$\forall\ b\in [\gamma,1]\ \exists\ i,j\in\{1,\ldots,l\}:\ |b-a_i|<\varepsilon\,,\ a_i\leq b\text{ and } |b-a_j|<\varepsilon\,,\ a_j\geq b\, . $$
Let $a'_i$ be the denominator of the reduced fraction of $a_i$. We define
$$k_i:=\max\{ l\in\sN:a'_i l\leq(1-\varepsilon)n \}.$$
By definition of $k_i$, for large $n$ it holds that
\begin{equation}\label{prop1}
(1-2\varepsilon)n<(1-\varepsilon)n-a'_i<a'_ik_i\leq(1-\varepsilon)n.
\end{equation}
Furthermore, for large $n$ and for all $i\in\{1,\ldots,l\}$ we have
\begin{equation}\label{prop1aa}
\tfrac{1}{a'_ik_i} \log E_\omega\big[ \eta_{a'_ik_i}(a'_ia_ik_i) \big]\geq\beta(a_i)-\varepsilon
\end{equation}
for \textsf{P}-a.e.\ $\omega$ due to \eqref{beta on Q}.\\
Now let $y\in n[\gamma,1]\cap\sN$. Then, there is $a_i\leq\tfrac{y}{n}$ with $|\tfrac{y}{n}-a_i|<\varepsilon$ and we have
\begin{equation}\label{prop1a}
a'_ia_ik_i\leq(1-\varepsilon)na_i\leq(1-\varepsilon)y\leq y.
\end{equation}
If $\beta(a_i)-\varepsilon\geq0$ due to \eqref{prop1}, \eqref{prop1aa} and \eqref{prop1a} we have
\bea
&&E_\omega\big[\eta_n(y)\big]\\
&\geq& E_\omega\big[\eta_{a'_ik_i}(a'_ia_ik_i)\big]\cdot E_\omega^{a'_ia_ik_i}\big[\eta_{n-a'_ik_i}(y)\big]\\
&\geq&\exp\big( a'_ik_i\cdot(\beta(a_i)-\varepsilon) \big)\cdot \delta_0^{n-a'_ik_i}\\
&=&\exp\big( \underbrace{a'_ik_i}_{\geq (1-2\varepsilon)n}\cdot\:(\beta(a_i)-\varepsilon) - \underbrace{(n-a'_ik_i)}_{\leq2\varepsilon n}\cdot\log(\delta_0^{-1})\big)\\
&\geq&\exp\Big(n \big( (1-2\varepsilon) \cdot(\beta(a_i)-\varepsilon) - 2\varepsilon \cdot\log(\delta_0^{-1}) \big)\Big)
\eea
for \textsf{P}-a.e.\ $\omega$ and for all large $n$, again with $\delta_0:=\delta^2\!\cdot\!(1-\delta)$. This yields for \textsf{P}-a.e.~$\omega$
\begin{eqnarray}\label{eqn100}
&&\tfrac{1}{n} \log E_\omega\big[\eta_n(y)\big]\nonumber\\
&\geq&(1-2\varepsilon)\cdot(\beta(a_i)-\varepsilon)-2\varepsilon\cdot\log(\delta_0^{-1}).
\end{eqnarray}
If $\beta(a_i)-\varepsilon<0$, we conclude in the same way that for \textsf{P}-a.e.\ $\omega$
\begin{eqnarray}\label{eqn101}
&&E_\omega\big[\eta_n(y)\big]\nonumber\\
&\geq&\exp\Big(n \big( (1-\varepsilon) \cdot(\beta(a_i)-\varepsilon) - 2\varepsilon \cdot\log(\delta_0^{-1}) \big)\Big).
\end{eqnarray}
Since $|a_i-\tfrac{y}{n}|<\varepsilon$ and since $\beta$ is uniformly continuous on $[\gamma,1]$, \eqref{eqn100} and \eqref{eqn101} imply \eqref{prop2} as $n\to\infty$ and $\varepsilon\to0$.
\\
\newline
(vi) To complete the proof we now have to show that for \textsf{P}-a.e.\ $\omega$
\begin{equation}\label{prop2a}
\limsup_{n\to\infty} \max_{x\in n[\gamma,1]\cap\sN}\Big( \tfrac{1}{n}\log E_\omega\big[\eta_n(x)\big]-\beta(\tfrac{x}{n}) \Big) \leq0.
\end{equation}
So we assume that \eqref{prop2a} does not hold and thus for infinitely many $n\in\sN$ there exists $y\in n[\gamma,1]\cap\sN$ such that
\begin{equation}\label{eqn102}
\tfrac{1}{n}\log E_\omega\big[\eta_{n}(y)\big]\geq\beta(\tfrac{y}{n})+\varepsilon
\end{equation}
holds with positive probability. As in (v), associated with $y$ there exists $a_j\geq\tfrac{y}{n}$ with $|\tfrac{y}{n}-a_j|<\varepsilon$.
We define
$$k'_j:=\max\{ l\in\sN:a'_jl\leq(1+\varepsilon)n \}.$$
Then \eqref{beta on Q} implies
\begin{equation}\label{prop3}
E_\omega\big[\eta_{a'_jk'_j}(a'_ja_jk'_j)\big]<\exp\big(a'_jk'_j\cdot(\beta(a_j)+\varepsilon)\big)
\end{equation}
for \textsf{P}-a.e.\ $\omega$ and for all large $n$. At the same time due to \eqref{eqn102} we have with positive probability
\bea
&&E_\omega\big[ \eta_{a'_jk'_j}(a'_ja_jk'_j)\big]\\
&\geq& E_\omega\big[\eta_{n}(y)\big]\cdot E_\omega^{y}\big[\eta_{a'_jk'_j-n}(a'_ja_jk'_j)\big]\\
&\geq&\exp\big(n(\beta(\tfrac{y}{n})+\varepsilon)\big)\cdot \delta_0^{a'_jk'_j-n}
\eea
since $a'_jk'_j-n>0$ and $a'_ja_jk'_j\geq (n + \varepsilon n-a'_j)a_j\geq na_j\geq y$ for large $n$. This yields a contradiction to \eqref{prop3} and hence completes the proof of the theorem.
\end{proof}

\begin{proof}[\bf Proof of Theorem \ref{existenz des limes}]
For any $\varepsilon > 0$ there exists $x_0 \in \mathbb{Q} \cap (0,1]$ such that
$$\beta(x_0) \geq \max_{x \in [0,1]} \beta(x) - \varepsilon.$$
Let $x_0' \in \sN$ be the denominator of the reduced fraction of $x_0$. Then we have for \textsf{P}-a.e.\ $\omega$
\bea
\liminf_{n \to \infty} \tfrac{1}{n x_0'} \log E_{\omega}\big[Z_{n x_0'}\big]&\geq& \liminf_{n \to \infty} \tfrac{1}{n  x_0'} \log E_{\omega}\big[\eta_{n x_0'}(n x_0'  \cdot x_0)\big]\\
&=& \beta(x_0) \geq \max_{x \in [0,1]} \beta(x) - \varepsilon
\eea
and because of the ellipticity condition \eqref{strong ellipticity}
$$ E_{\omega} \big[Z_{n x_0' + r}\big] \geq \delta_0^r \cdot E_{\omega} \big[Z_{n x_0'}\big]$$
for $r \in \{0,1,\ldots,x_0'-1\}$ and for \textsf{P}-a.e.\ $\omega$. We conclude for $\varepsilon \to 0$ that for \textsf{P}-a.e.~$\omega$
\begin{equation}\label{liminf1}
\liminf_{n \to \infty} \tfrac{1}{n} \log E_{\omega} \big[Z_n\big] \geq \max_{x \in [0,1]} \beta(x).
\end{equation}
To get the other inequality we first state the following
\begin{lem} \label{lem1}
For $\varepsilon > 0$ there is $\gamma > 0$ such that for all $n\in\sN$ we have
$$\tfrac{1}{n} \log E_{\omega}\big[\eta_n(y)\big] \leq \log (\Lambda + \varepsilon)\quad\textnormal{for \textsf{P}-a.e. }\omega$$
for all $y\in n[0,\gamma]\cap\sN_0$.
\end{lem}
\renewcommand{\qedsymbol}{$\square$}
\begin{proof}[Proof of Lemma \ref{lem1}]
For $\tfrac{1}{2} > \gamma > 0$ and $y < \gamma n$ we have
$$E_{\omega}[\eta_n(y)] \leq \tbinom{n}{y} \cdot\Lambda^{n-y} \cdot M^y\quad\textnormal{for \textsf{P}-a.e. }\omega.$$
Since
$$ \tfrac{1}{n} \log \tbinom{n}{y} \leq \tfrac{1}{n} \log \tbinom{n}{\left\lfloor \gamma n \right\rfloor} 
\to 0$$
for $\gamma \to 0$ uniformly in $n$, we get for \textsf{P}-a.e.\ $\omega$
$$\tfrac{1}{n} \log E_{\omega} [\eta_n(y)] \leq o(\gamma) + \tfrac{n-y}{n} \log (\Lambda) + \tfrac{y}{n} \log(M) \leq \log(\Lambda + \varepsilon)$$
for $\gamma >0$ small enough.
\end{proof}
For an arbitrary $\varepsilon > 0$ we now choose $\gamma>0$ as in Lemma \ref{lem1}. Then, by Theorem~\ref{prop local growth} and Lemma~\ref{lem1} we get
\bea
&&\limsup_{n\to\infty}\tfrac{1}{n} \log E_{\omega} [Z_n]\\
&=& \limsup_{n\to\infty}\tfrac{1}{n} \log E_{\omega} \left( \sum_{y=0}^{\lfloor\gamma n\rfloor-1} \eta_n(y) + \sum_{y=\lfloor\gamma n\rfloor}^{n} \eta_n(y)\right)\\
&\leq&\limsup_{n\to\infty}\tfrac{1}{n} \log\left(\gamma n\! \cdot\! \big(\Lambda + \varepsilon\big)^n + n\! \cdot\! \exp\Big(n\!\cdot\!\big(\max_{x \in [0,1]} \beta(x) + o(n)\big)\Big)\right)\\
&\leq&\max_{x \in [0,1]} \beta(x) + \varepsilon
\eea
for \textsf{P}-a.e.\ $\omega$ since $\beta(0)=\log(\Lambda)$. For $\varepsilon \to 0$ this yields for \textsf{P}-a.e.\ $\omega$
$$\limsup_{n \to \infty} \tfrac{1}{n} \log E_{\omega} [Z_n] \leq \max_{x \in [0,1]} \beta(x)\, ,$$
which, together with \eqref{liminf1}, proves the claim. 
\renewcommand{\qedsymbol}{$\blacksquare$}
\end{proof}

\begin{proof}[\bf Proof of Theorem \ref{GS und EW}]
We start by proving that \textit{(ii)} implies \textit{(i)}.\\
First assume that there is \textsf{P}-a.s.\ LS. As shown in the proof of Theorem \ref{LS} for \textsf{P}-a.e.\ $\omega$ there is a location $x$ such that the descendants of a particle at $x$ that stay at $x$ form a supercritical Galton-Watson process. Let $x=x(\omega)$ be such a location, i.e.\ $m_x(1-h_x)>1$. Then we have for \textsf{P}-a.e.\ $\omega$ and for $n\geq x$
\bea
E_\omega[Z_n]&\geq&E_\omega[\eta_n(x)]\\
 &\geq&\big(1-\mu_0\big(\{0\}\big)\big)h_0\cdot\ldots\cdot\big(1-\mu_{x-1}\big(\{0\}\big)\big)h_{x-1} \cdot\big(m_x(1-h_x)\big)^{n-x}\\
 &\geq&(\delta^{2x} \cdot \big(m_x(1-h_x)\big)^{n-x}
\eea
where we used condition \eqref{ellipticity} for the last inequality. Due to Theorem~\ref{existenz des limes} we obtain for \textsf{P}-a.e.\ $\omega$
\bea
\lim_{n\to\infty} \tfrac{1}{n}\log E_\omega[Z_n]&\geq&\limsup_{n\to\infty}\tfrac{1}{n}\log\big(\delta^{2x} \cdot (m_x(1-h_x))^{n-x}\big)\\
 &=&\log\left(m_x(1-h_x)\right)\\
 &>&0.
\eea
Now let us assume that there is \textsf{P}-a.s.\ no LS, which is according to Theorem~\ref{LS} equivalent to $\Lambda\leq1$. Again, we use the process $(\xi_n)_{n\in\sN_0}$ defined in the proof of Theorem~\ref{GS}. \\
Since there is GS for \textsf{P}-a.e.\ $\omega$, the process $(\xi_n)_{n\in\sN_0}$ has a positive probability of survival for \textsf{P}-a.e. $\omega$. Thus we have
\begin{equation}\label{gleichung tanny}
\int \log \big( E_\omega[\xi_1]\big)\textsf{P}(d\omega)>0
\end{equation}
by Theorem 5.5 in \cite{tanny}. For $T\in\sN$ we now introduce a slightly modified embedded branching process $(\xi_n^T)_{n\in\sN_0}$. For $k\in\sN$ we define $\xi_k^T$ as the total number of all particles that move from position $k-1$ to $k$ within $T$ time units after they were released at position $k-1$. The left over particles are no longer considered. With $\xi^T_0:=1$ we observe that $(\xi_n^T)_{n\in\sN_0}$ is a branching process in an i.i.d.\ environment. By the monotone convergence theorem and \eqref{gleichung tanny} there exists some $T$ such that
\begin{equation}\label{EW groesser 0}
\int \log\left( E_\omega\big[\xi_1^T\big] \right)\textsf{P}(d\omega)>0.
\end{equation}
By construction of $(\xi_n^T)_{n\in\sN_0}$ we obtain
\begin{equation}\label{abschaetzung xi^T und Zn}
\xi_n^T\ \leq\ Z_n+Z_{n+1}+\ldots+Z_{nT}.
\end{equation}
Using the strong law of large numbers and taking into account that $\omega$ is an i.i.d. sequence, we have
\begin{eqnarray}\label{LLN}
&&\lim_{n\to\infty}\tfrac{1}{n}\log E_\omega\big[\xi_n^T\big] \nonumber \\
&=&\lim_{n\to\infty}\tfrac{1}{n}\log \prod_{i=0}^n E_{\theta^n\omega} \big[\xi_1^T\big] \nonumber \\
&=&\lim_{n\to\infty}\tfrac{1}{n}\sum_{i=0}^n\log E_{\theta^n\omega} \big[\xi_1^T\big] \nonumber \\
&=&\int \log \left( E_\omega\big[\xi_1^T\big] \right) \textsf{P}(d\omega)\quad\textnormal{for \textsf{P}-a.e. }\omega.
\end{eqnarray}
Here $\theta$ again denotes the shift operator as usual, i.e.\ $(\theta\,\omega)_i = \omega_{i+1}$. Together with \eqref{EW groesser 0} and \eqref{abschaetzung xi^T und Zn} this yields for \textsf{P}-a.e.\ $\omega$
\begin{equation}\label{gleichung summe der Z_n}
\liminf_{n\to\infty}\tfrac{1}{n}\log E_\omega\big[Z_n+Z_{n+1}+\ldots+Z_{n T}\big]>0.
\end{equation}
Now we conclude using Theorem~\ref{existenz des limes} that for \textsf{P}-a.e.\ $\omega$
$$\max_{x\in[0,1]}\beta(x) = \lim_{n\to\infty}\tfrac{1}{n}\log E_\omega[Z_n]> 0$$
because otherwise there would be a contradiction to \eqref{gleichung summe der Z_n}. This shows that \textit{(ii)} implies \textit{(i)}.\vspace{12pt}\\
To show that \textit{(i)} implies \textit{(ii)} we first notice that \textit{(ii)} obviously holds if there is LS for \textsf{P}-a.e.\ $\omega$.
Therefore we may assume $\Lambda~\leq~1$ for the rest of the proof.\\
Now label every particle of the entire branching process and let $\Gamma$ denote the set of all produced particles. Write $ \sigma \prec \tau $ for two particles $\sigma \neq \tau$ if $\sigma$ is an ancestor of $\tau$ and denote by $|\sigma|$ the generation in which the particle $\sigma$ is produced. Furthermore, for every $\sigma \in \Gamma$ let $X_{\sigma}$ be the random location of the particle $\sigma$. Using these notations we define
\begin{equation} \label{mengeg}
G_i := \{ \tau \in \Gamma:\, X_{\tau} = i,\, X_{\sigma} < i\text{ for all }\sigma \in \Gamma,\, \sigma \prec \tau \}
\end{equation}
for every $i \in\sN_0$. Therefore $G_i$ is for $i\neq0$ the set of all the particles $\tau$ that move from position $i-1$ to position $i$ and hence the particles in $G_i$ are the first particles at position $i$ in their particular ancestral lines. We observe that the process $(|G_n|)_{n\in\sN_0}$ coincides with $(\xi_n)_{n\in\sN_0}$. Further, define for every $\sigma \in \Gamma$ and $n \in\sN_0$
$$H_n^{\sigma}:= \left| \{ \tau \in \Gamma: \, \sigma \preceq \tau, \, |\tau| = n, \, X_{\tau} = X_{\sigma} \} \right|$$
as the number of descendants of the particle $\sigma$ in generation $n$ which are still at the same location as the particle $\sigma$. This enables us to decompose $Z_n$ in the following way:
\begin{equation}\label{zer1}
Z_n = \sum_{i=1}^{n} \sum_{\sigma \in G_i} H_{n - |\sigma|}^{\sigma} 
\end{equation}
Since by assumption there is no LS, we have for \textsf{P}-a.e.\ $\omega$
\begin{equation}\label{zer2}
E_{\omega} [H_n^{\sigma} \, | \, \sigma \in \Gamma,\, X_{\sigma} = i] \leq 1 
\end{equation}
because for any existing particle $\sigma$ its progeny which stays at the location of $\sigma$ forms a Galton-Watson process which eventually dies out. By \eqref{zer1} and \eqref{zer2} we conclude that for \textsf{P}-a.e.\ $\omega$ we have
$$E_{\omega}[Z_n] \leq \sum_{i=1}^{n} E_{\omega} [|G_i|].$$
Therefore due to $(i)$ we get
$$\limsup_{n \to \infty} \tfrac{1}{n} \log E_{\omega} [|G_n|] > 0  \quad\text{ for }\textsf{P}\text{-a.e. }\omega.$$
Since $(|G_n|)_{n\in\sN_0}$ coincides with the branching process in random environment $(\xi_n)_{n\in\sN_0}$, we obtain
$$\int\, \log\big( E_\omega[\xi_1]\big)\, \textsf{P}(d\omega) = \lim_{n \to \infty} \tfrac{1}{n} \log E_{\omega} [|G_n|] > 0 \quad\text{ for }\textsf{P}\text{-a.e. }\omega$$
as in \eqref{LLN}. But then again, we have GS for \textsf{P}-a.e.\ $\omega$ since $(\xi_n)_{n\in\sN_0}$ survives with positive probability for \textsf{P}-a.e.\ $\omega$.
\end{proof}

\begin{proof}[\bf Proof of Theorem \ref{thm GS und Zn}]
In this proof we use the expression ``a.s.'' in the sense of ``$P_\omega$-a.s.\ for \textsf{P}-a.e.\ $\omega$''.
\vspace{12pt}
\\
{\bf Part 1.} In the first part of the proof we show in three steps that we have a.s.
\begin{equation} \label{eq28}
\limsup_{n \to \infty} \tfrac{1}{n} \log Z_n \leq \max_{x \in [0,1]} \beta(x).
\end{equation}
(i) To obtain \eqref{eq28} we start by showing that for all $\gamma>0$ we have a.s.
\begin{equation}\label{local growth Zn}\limsup_{n\to\infty}\ \max_{x\in n[\gamma,1]\cap\sN}\left( \tfrac{1}{n}\log \eta_n(x) -\beta(\tfrac{x}{n}) \right)\leq 0.\end{equation}
To see this fix $\gamma>0$ and $\varepsilon>0$.\\
Then, by Theorem \ref{prop local growth} for \textsf{P}-a.e.\ $\omega$ there exists $N=N(\omega, \gamma, \varepsilon)$ such that for all $n\geq N$ and for all $y\in n[\gamma,1]\cap\sN$ we have
$$E_\omega[\eta_n(y)]\leq\exp\big(n\cdot(\beta(\tfrac{y}{n})+\varepsilon)\big).$$
Thus, for \textsf{P}-a.e. $\omega$ we obtain for large $n$ and for all $y\in n[\gamma,1]\cap\sN$
$$P_\omega\Big( \eta_n(y)\geq\exp\big( n\cdot(\beta(\tfrac{y}{n})+2\varepsilon) \big) \Big)\leq \frac{ E_\omega[\eta_n(y)] }{ \exp(n\cdot(\beta(\tfrac{y}{n})+2\varepsilon)) }=\exp(-\varepsilon n).$$
Using the Borel-Cantelli lemma and taking into account that $\big|n[\gamma,1]\cap\sN\big|\leq n$ this yields that a.s.\ we have
$$\limsup_{n \to \infty} \max_{x\in n[\gamma,1]\cap\sN}\big(\tfrac{1}{n}\log\eta_n(y)-\beta(\tfrac{y}{n})\big)<2\varepsilon.$$
Since $\varepsilon$ is arbitrarily small, this proves \eqref{local growth Zn}.
\vspace{12pt}
\\
(ii) Secondly, we show that for every $\varepsilon>0$ there exists $\gamma=\gamma(\varepsilon)>0$ such that a.s.\ we have
\begin{equation} \label{eq5.62}
\limsup_{n\to\infty}\max_{x\in n[0,\gamma]\cap\sN}\big( \tfrac{1}{n}\log\eta_n(x)-\beta(0)-\varepsilon  \big)\leq 0.
\end{equation}
To see this we observe that according to Lemma \ref{lem1} for every $\varepsilon>0$ there is $\gamma = \gamma(\varepsilon)>0$ such that
\bea
\tfrac{1}{n} \log E_{\omega}\big[\eta_n(y)\big] \leq \log (\Lambda + \varepsilon) \stackrel{(\Lambda>1)}{\leq} \log(\Lambda) + \varepsilon = \beta(0) + \varepsilon
\eea
for \textsf{P}-a.e.\ $\omega$ and for $0 \leq y \leq \gamma n$. Therefore the same argument as in (i) yields \eqref{eq5.62}.
\vspace{12pt}
\\
(iii) We now combine (i) and (ii) to obtain \eqref{eq28}. For an arbitrary $\varepsilon>0$ choose $\gamma >0$ as in (ii). Then \eqref{local growth Zn} and \eqref{eq5.62} imply that a.s.\ we have
\bea
&& \limsup_{n \to \infty} \tfrac{1}{n} \log Z_n\\
&=&\limsup_{n \to \infty} \tfrac{1}{n} \log \left( \sum_{y=0}^{\lfloor\gamma n\rfloor-1} \eta_n(y) + \sum_{y=\lfloor\gamma n\rfloor}^{n} \eta_n(y)\right)\\
&\leq&\limsup_{n \to \infty} \tfrac{1}{n} \log\left(\gamma n\! \cdot\! \exp\big(n\!\cdot\!(\beta(0) + \varepsilon)\big) + n\! \cdot\! \exp \Big(n\! \cdot\! \big(\max_{x \in [0,1]} \beta(x) + o(n)\big)\Big)\right)\\
&\leq& \max_{x \in [0,1]} \beta(x) + \varepsilon.
\eea
For $\varepsilon \to 0$ this implies \eqref{eq28} and thus the first part of the proof is complete.
\vspace{12pt}
\\
{\bf Part 2.} In the second part of the proof we show that
\begin{equation}\label{eq27}
\left.P_{\omega} \left( \liminf_{n \to \infty} \tfrac{1}{n} \log Z_n \geq \max_{x\in[0,1]}\beta(x)\ \right|\ Z_n\not\to 0  \right)=1 \quad \text{for }  \textnormal{\textsf{P}-a.e. } \omega.
\end{equation}
We start by stating the following
\begin{lem} \label{lem2}
For all $\varepsilon>0$ and $r,s\in\sN$ with $r\leq s$ and $\beta(\tfrac{r}{s})-\varepsilon>0$ there exists $N_0 \in \sN$ such that for \textnormal{\textsf{P}}-a.e. $\omega$ we have
$$P_\omega\left(\liminf_{n\to\infty} \tfrac{1}{nsN_0}\log\eta_{nsN_0}(nrN_0)\geq\beta(\tfrac{r}{s})-\varepsilon\right)>0.$$
\end{lem}
\begin{proof}
Define
$$M_N:=\left\{\omega\in\Omega:\tfrac{1}{sN}\log E_\omega[\eta_{sN}(rN)]\geq\beta(\tfrac{r}{s})-\tfrac{\varepsilon}{2} \right\}.$$
Then, for every $\varepsilon_0>0$ there exists $N_0=N_0(\varepsilon_0)$ such that
$$\textnormal{\textsf{P}}\big(M_{N_0}\big)\geq1-\varepsilon_0$$
and thus for sufficiently small $\varepsilon_0$ and the corresponding $N_0(\varepsilon_0)$ we have
\begin{eqnarray}
\label{lemma supercrit}
&&\int \log E_\omega \big[\eta_{sN_0}(rN_0)\big] \textnormal{\textsf{P}}(d\omega)\nonumber\\
&\geq& sN_0(\beta(\tfrac{r}{s})-\tfrac{\varepsilon}{2})(1-\varepsilon_0)\nonumber\\
&\geq& sN_0(\beta(\tfrac{r}{s})-\varepsilon) + sN_0(\tfrac{\varepsilon}{2}-\beta(\tfrac{r}{s})\varepsilon_0 + \tfrac{\varepsilon}{2} \varepsilon_0)\nonumber\\
&\geq& sN_0(\beta(\tfrac{r}{s})-\varepsilon)\ >\ 0.
\end{eqnarray}
We now construct a branching process in random environment $(\psi_n)_{n\in\sN_0}$ which is dominated by $\big(\eta_{nsN_0}(nrN_0)\big)_{n\in\sN_0}$. After starting with one particle at $0$ we count all the particles that are at time $sN_0$ at position $rN_0$ and denote this number by $\psi_1$. The remaining particles are removed from the system and no longer considered. After that we count the number of particles at time $2sN_0$ at position $2rN_0$ and denote this number by $\psi_2$. Repeating this procedure yields the process $(\psi_n)_{n\in\sN_0}$ which is supercritical due to \eqref{lemma supercrit}. In fact \eqref{lemma supercrit} and Theorem~5.5 in \cite{tanny} now imply that for sufficiently small $\varepsilon_0$,
\begin{equation} \label{expgr}
\liminf_{n\to\infty} \tfrac{1}{n}\log \eta_{nsN_0}(nrN_0)  \geq sN_0(\beta(\tfrac{r}{s})-\varepsilon)
\end{equation}
a.s.\ on $\{\psi_n\not\to0\}$. Since we assume condition \eqref{strong ellipticity}, Corollary 6.3 in \cite{tanny} implies
\begin{equation}\label{wkeit1}
P_{\omega}(\psi_n \to 0) < 1
\end{equation}
for \textsf{P}-a.e.\ $\omega$.
Combining \eqref{expgr} and \eqref{wkeit1} now completes the proof of the lemma.
\renewcommand{\qedsymbol}{$\square$}
\end{proof}
Lemma \ref{lem2} yields the following
\begin{cor} \label{cor1}
Let $\varepsilon,r,s$ and $N_0$ be as in Lemma~\ref{lem2}. Then there exists $\nu>0$ such that for \textnormal{\textsf{P}}-a.e. $\omega$ there exists an increasing sequence $(x_l)_{l\in\sN_0}=(x_l(\omega))_{l\in\sN_0}$ in $\sN_0$ such that for all $l\in\sN_0$ we have
$$P_{\omega}^{x_l}\left(\liminf_{n\to\infty} \tfrac{1}{nsN_0}\log\eta_{nsN_0}(nrN_0)\geq\beta(\tfrac{r}{s})-\varepsilon\right)>\nu.$$
\end{cor}
\begin{proof}
Due to Lemma \ref{lem2} there exists $\nu > 0$ such that
$$ \textnormal{\textsf{P}}\left(\left\{ \omega: \, P_\omega\left(\liminf_{n\to\infty} \tfrac{1}{nsN_0}\log\eta_{nsN_0}(nrN_0)\geq\beta(\tfrac{r}{s})-\varepsilon\right)>\nu \right\}\right) > 0. $$
Since the sequence
\bea
&&\bigg(P^x_\omega\left(\liminf_{n\to\infty} \tfrac{1}{nsN_0}\log\eta_{nsN_0}(nrN_0)\geq\beta(\tfrac{r}{s})-\varepsilon\right) \bigg)_{x \in \sN_0}\\
&=&\bigg(P_{\theta^{x} \omega}\left(\liminf_{n\to\infty} \tfrac{1}{nsN_0}\log\eta_{nsN_0}(nrN_0)\geq\beta(\tfrac{r}{s})-\varepsilon\right)\bigg)_{x\in\sN_0}
\eea
is ergodic with respect to \textnormal{\textsf{P}}, the ergodic theorem yields
$$\lim_{n\to\infty}\tfrac{1}{n}\sum_{k=0}^{n-1} \mathds{1}\Big\{ P_{\theta^{x} \omega}\big(\liminf_{n\to\infty} \tfrac{1}{nsN_0}\log\eta_{nsN_0}(nrN_0)\geq\beta(\tfrac{r}{s})-\varepsilon\big)>\nu\Big\}>0$$
for \textsf{P}-a.e.\ $\omega$ and this completes the proof of the corollary.
\renewcommand{\qedsymbol}{$\square$}
\end{proof}
\renewcommand{\qedsymbol}{$\blacksquare$}
Let $(x_l)_{l\in\sN_0}$ be an increasing sequence of positions as in Corollary~\ref{cor1}. We now show in two steps that a.s.\ on the event of non-extinction there will eventually be some particle at one of the positions $x_l$ such that the descendants of this particle constitute a process with the desired growth.\vspace{12pt}\\
(i) As a first step we show that a.s.\ on the event of survival $(Z_n)_{n\in\sN_0}$ grows as desired along some subsequence $(j+nsN_0)_{n\in\sN_0}$ for some $j\in\{0,\ldots,sN_0-1\}$. To obtain this, as in the proof of Theorem~\ref{GS und EW}, let $\Gamma$ again denote the set of all existing particles and for $\sigma \in \Gamma$ let $\eta_n^{\sigma}(y)$ denote the number of descendants of $\sigma$ among the particles which belong to $\eta_n(y)$. With the sets $(G_l)_{l \in \sN_0}$ as in \eqref{mengeg} and the sequence $(x_l)_{l \in \sN_0}$ as in Corollary \ref{cor1} we define:
\bea
A_{x_l}&:=& \Big\{ \exists\ \sigma \in G_{x_l}: \liminf_{n \to \infty} \tfrac{1}{nsN_0}\log\eta_{\left|\sigma\right|+nsN_0}^{\sigma}(x_l + nrN_0)\geq\beta(\tfrac{r}{s})-\varepsilon  \Big\}
\\ B_{x_l} &:=& \Big\{|G_{x_l}| \geq l \Big\}
\eea
Due to Corollary~\ref{cor1} and since the descendants of all particles belonging to $G_{x_l}$ evolve independently we get
$$ P_{\omega} \left( A_{x_l}^c \cap B_{x_l} \right) \leq (1-\nu)^l \quad \text{for }  \textnormal{\textsf{P}-a.e. } \omega $$
and therefore we conclude with the Borel-Cantelli lemma that
\begin{equation}\label{limsup1}
P_{\omega} \left( \limsup_{l \to \infty} \left( A_{x_l}^c \cap B_{x_l} \right) \right) = 0 \quad \text{for }  \textnormal{\textsf{P}-a.e. }  \omega.
\end{equation}
According to Theorem 5.5 of \cite{tanny} we have a.s.\ exponential growth of the process $\big(|G_l|\big)_{l \in \sN_0}$ on the event of survival and therefore it holds that we have a.s.
$$  \liminf_{l \to \infty} B_{x_l}  = \big\{Z_n\not\to 0\big\}.$$
Together with \eqref{limsup1} this yields
$$\left. P_{\omega} \left( \limsup_{l \to \infty} A_{x_l}^c\ \right|\ Z_n\not\to 0\right) = 0 \quad \text{for }  \textnormal{\textsf{P}-a.e. } \omega.$$
Thus a.s.\ on $\{Z_n\not\to0\}$ there is $l\in\sN_0$ and $\sigma\in G_{x_l}$ such that
$$\liminf_{n \to \infty} \tfrac{1}{nsN_0}\log\eta_{\left|\sigma\right|+nsN_0}^{\sigma}(x_l + nrN_0)\geq\beta(\tfrac{r}{s})-\varepsilon$$
and hence we have for \textsf{P}-a.e. $\omega$
\begin{eqnarray}\label{eq26}
&&P_{\omega}\left(\bigcup_{\sigma\in\Gamma}\left\{\liminf_{n \to \infty} \tfrac{1}{nsN_0}\log Z_{\left|\sigma\right|+nsN_0}\geq\beta(\tfrac{r}{s})-\varepsilon\right\}\ \Bigg|\ Z_n\not\to0\right)\nonumber\\
&=&P_{\omega}\left(\bigcup_{j\in\sN_0}\left\{\liminf_{n \to \infty} \tfrac{1}{nsN_0}\log Z_{j+nsN_0}\geq\beta(\tfrac{r}{s})-\varepsilon\right\}\ \Bigg|\ Z_n\not\to0\right)\nonumber\\
&=&P_{\omega}\left(\bigcup_{j=1}^{sN_0}\left\{\liminf_{n \to \infty} \tfrac{1}{nsN_0}\log Z_{j+nsN_0}\geq\beta(\tfrac{r}{s})-\varepsilon\right\}\ \Bigg|\ Z_n\not\to0\right)\ =\ 1.\hspace{20pt}
\end{eqnarray}
(ii) The last step of this part of the proof is to show that the growth along some subsequence $(j+nsN_0)_{n\in\sN_0}$ already implies sufficiently strong growth of $(Z_n)_{n\in\sN_0}$.\\
Due to the ellipticity condition \eqref{strong ellipticity} we have (recalling $\delta_0 = \delta^2(1-\delta)$)
$$ P_{\omega}^x\big(\eta_i(x) \geq 1\big) \geq \delta_0^i \quad \text{for all } i,x \in \sN_0.$$
A large deviation bound for the binomial distribution therefore implies
\begin{equation}\label{large deviation}
\left.P_{\omega} \left(Z_{n+i} \leq Z_n \cdot \tfrac{\delta_0^i}{2}\ \right|\ Z_n=m \right) \leq \exp(-m \cdot \lambda_0) \quad \forall\ m \in \sN
\end{equation}
for $i \in \{1,...,sN_0\}$ and some constant $\lambda_0 = \lambda_0(N_0) >0$. We now define:
\bea
C_{j,n}&:=& \bigcup_{i=1}^{sN_0}\left\{Z_{j+nsN_0+i}\leq\tfrac{\delta_0^{sN_0}}{2}\exp\big(nsN_0\!\cdot\!(\beta(\tfrac{r}{s})-\varepsilon)\big)\right\}\\
D_{j,n}&:=& \left\{\tfrac{1}{nsN_0}\log Z_{j+nsN_0}\geq\beta(\tfrac{r}{s})-\varepsilon\right\}
\eea
Then due to \eqref{large deviation} for \textsf{P}-a.e.\ $\omega$ we have for all $j\in\{1,\ldots,sN_0\}$
\begin{eqnarray}\label{BC large deviation}
&&P_\omega\left( C_{j,n}\cap D_{j,n} \right)\nonumber\\
&\leq& sN_0\!\cdot\!\exp\big( -\lambda_0 \exp( n\!\cdot\!\lambda_1) \big)
\end{eqnarray}
where $\lambda_1:=sN_0\!\cdot\!(\beta(\tfrac{r}{s})-\varepsilon)$.\\
Since the upper bound in \eqref{BC large deviation} is summable in $n\in\sN_0$, we can apply the Borel-Cantelli lemma and conclude that for \textsf{P}-a.e.\ $\omega$ we have for all $j\in\{1,\ldots,sN_0\}$
\bea
&&\left.P_\omega\left(\limsup_{n\to\infty}C_{j,n}\ \right|\ \liminf_{n\to\infty} D_{j,n} \right)\\
&\leq& P_{\omega}\left(\liminf_{n\to\infty} D_{j,n} \right)^{-1} \cdot \; P_\omega\left( \limsup_{n\to\infty} (C_{j,n}\cap D_{j,n}) \right)\ =\ 0
\eea
Thus for \textsf{P}-a.e.\ $\omega$ we have for all $j\in\{1,\ldots,sN_0\}$
$$\left.P_\omega\left( \liminf_{n\to\infty}\tfrac{1}{n}\log Z_n\leq \beta(\tfrac{r}{s})-2\varepsilon\ \right|\ \liminf_{n\to\infty} D_{j,n} \right)=0$$
and this implies
\begin{equation}
\label{eq34}
P_\omega\left( \liminf_{n\to\infty}\tfrac{1}{n}\log Z_n\leq \beta(\tfrac{r}{s})-2\varepsilon\ \left|\ \bigcup_{j=1}^{sN_0}\liminf_{n\to\infty} D_{j,n} \right)\right.=0.
\end{equation}
Using \eqref{eq26} and \eqref{eq34} we obtain
\bea
&&\left.P_\omega\left(\liminf_{n\to\infty}\tfrac{1}{n}\log Z_n\leq \beta(\tfrac{r}{s})-2\varepsilon\ \right|\ Z_n\not\to0 \right)\\
&\leq&P_\omega\left(Z_n\not\to0\right)^{-1} \cdot P_\omega\left(\left\{ \liminf_{n\to\infty}\tfrac{1}{n}\log Z_n\leq\beta(\tfrac{r}{s})-2\varepsilon\right\} \cap \bigcup_{j=1}^{sN_0}\liminf_{n\to\infty}D_{j,n} \right)\\
&=&0
\eea
which yields
\begin{equation}\label{eq270}
\left.P_{\omega} \left( \liminf_{n \to \infty} \tfrac{1}{n} \log Z_n > \beta(\tfrac{r}{s}) - 2 \varepsilon\ \right|\ Z_n\not\to 0  \right)=1 \quad \text{for }  \textnormal{\textsf{P}-a.e. } \omega.
\end{equation}
Since $r$ and $s$ can be chosen such that $\beta(\tfrac{r}{s})$ is arbitrarily close to $\max_{x \in [0,1]} \beta(x)$, \eqref{eq270} implies \eqref{eq27} as $\varepsilon\to0$ and the proof is complete.\vspace{12pt}\\
\end{proof}
\begin{proof}[\bf Proof of Theorem \ref{phase transitions}]
If $M\leq1$ then
$$\log\left(\frac{m_0h}{1-m_0(1-h)}\right)\leq0\quad\textnormal{\textsf{P}}\text{-a.s.}$$
and therefore Theorem \ref{GS} implies \textit{(i)}.\vspace{12pt}\\
We continue with proving \textit{(ii)} and assume that $M>1$. If $m_0=M$ \textnormal{\textsf{P}}-a.s.\ then
$$\log\left(\frac{m_0h}{1-m_0(1-h)}\right)>0\quad\textnormal{\textsf{P}}\text{-a.s.}$$
and thus $\varphi(h)>0$ for all $h\in(h_{LS},1]$. This case is included in \textit{(c)}.\\
In the following we assume that $m_0$ is not deterministic. We notice that $\varphi$ is finite and continuously differentiable for $h\in\left(h_{LS},1\right]$ since
$$\frac{\partial}{\partial h}\log\left(\frac{m_0h}{1-m_0(1-h)}\right)=\frac{1}{h}-\frac{m_0}{1-m_0(1-h)}$$
is a.s.\ uniformly bounded for $h\in[h_{LS}+\varepsilon,1]$ with $\varepsilon>0$. Thus we have
\begin{equation}\label{abl}
\frac{\partial}{\partial h}\varphi(h)=\textnormal{\textsf{E}}\left[\frac{1}{h}-\frac{m_0}{1-m_0(1-h)}\right].
\end{equation}
Now assume that there exists $h^*\in(h_{LS},1]$ with $\varphi(h^*)=0$. Then
\begin{equation}\label{logeq}
\textnormal{\textsf{E}}\left[\log\left(\frac{m_0}{1-m_0(1-h^*)}\right)\right]= \log\left(\frac{1}{h^*}\right)\, .
\end{equation}
Due to the strict concavity of $y\longmapsto \log y$ we have
\begin{equation}\label{jensen}
\log\left(\textnormal{\textsf{E}}\left[\frac{m_0}{1-m_0(1-h^*)}\right]\right)>\log\left(\frac{1}{h^*}\right)
\end{equation}
by Jensen's inequality and (\ref{logeq}). Thus we obtain that $\varphi$ is strictly decreasing in $h=h^*$ by \eqref{abl} and \eqref{jensen}.\\
Now assume $\varphi(h_{LS})=0$. As above Jensen's inequality yields \eqref{jensen} for $h_{LS}$ instead of $h^*$. Since the mapping
$$h\longmapsto\frac{m_0}{1-m_0(1-h)}$$
is decreasing in $h>1-\frac{1}{m_0}$, we have
$$\lim_{\varepsilon\downarrow0}\, \textnormal{\textsf{E}}\left[\frac{m_0}{1-m_0(1-h_{LS}+\varepsilon)}\right]=\textnormal{\textsf{E}}\left[\frac{m_0}{1-m_0(1-h_{LS})}\right]>\frac{1}{h_{LS}}$$
by the monotone convergence theorem. Thus $\varphi$ is strictly decreasing and therefore negative in $h\in(h_{LS},h_{LS}+\varepsilon)$ for some sufficiently small $\varepsilon>0$.
\vspace{12pt}\\
Now we obtain \textit{(a) -- (c)} by the continuity of $\varphi$ and the fact that $\varphi$ is strictly decreasing in every zero in $[h_{LS},1]$.
\end{proof}

\section{Examples}\label{examples}
\textbf{1.}\quad A basic and natural example to illustrate our results is the following. Let $\mu^{(+)}$ and $\mu^{(-)}$ be two different non-trivial offspring distributions. We define
$$m^{(+)}:=\sum_{k=0}^\infty k\,\mu^{(+)}\!(k)\quad\text{ and }\quad m^{(-)}:=\sum_{k=0}^\infty k\,\mu^{(-)}\!(k)$$
and suppose
$$0<m^{(-)}<m^{(+)}\leq\infty.$$
\begin{figure}[ht]
 \begin{minipage}[t]{.38\linewidth} 
\includegraphics [viewport=100 660 210 820, scale=0.9]{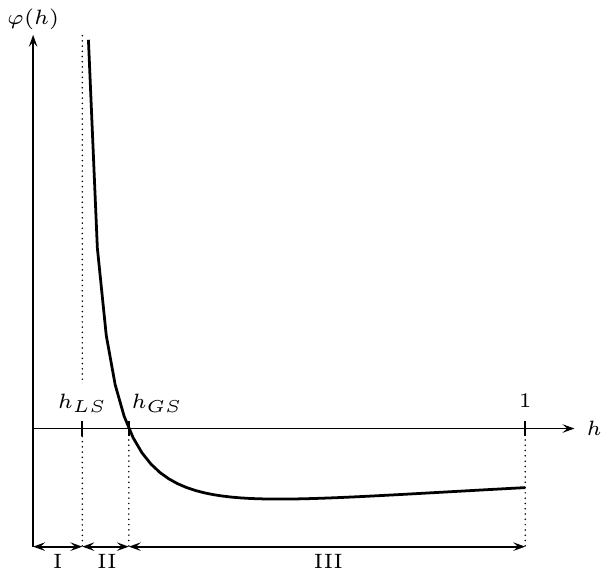}
  \caption{There are three regimes: I: LS, II: GS but no LS, III: no GS}
  \end{minipage}
  \hspace{0.1\linewidth} 
  \begin{minipage}[t]{.38\linewidth} 
\includegraphics [viewport=100 660 210 820, scale=0.9]{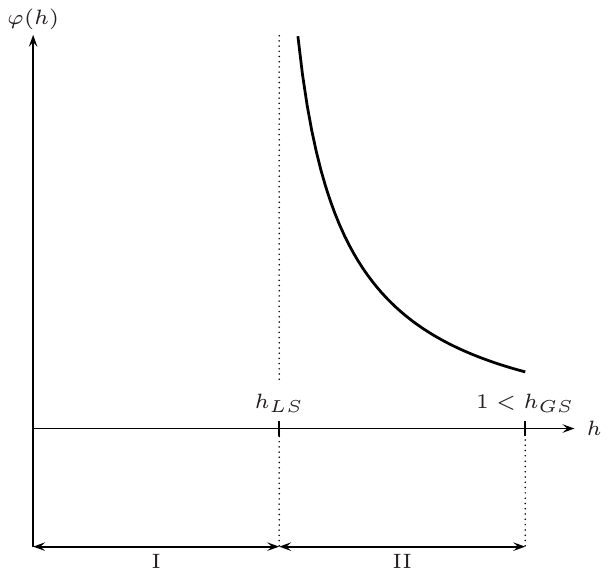}
  \caption{There are two regimes: I: LS, II: GS but no LS}
  \end{minipage}
\end{figure}\\
Furthermore, let
$$\textsf{P}\left(\mu_0 = \mu^{(+)}\right) = 1-\textsf{P}\left(\mu_0=\mu^{(-)}\right)=q\in(0,1).$$
This setting obviously satisfies condition \eqref{ellipticity}.
For figures 1 and 2 we have chosen 
\begin{alignat*}{3}
&q=\frac{3}{4},&\qquad &m^{(+)}=\frac{10}{9},&\qquad &m^{(-)}=\frac{2}{5},\\
&q=\frac{1}{2},&\qquad &m^{(+)}=\,\,2,&\qquad &m^{(-)}=\frac{2}{3},
\end{alignat*}
respectively.
\vspace{12pt}\\
\textbf{2.}\quad As already announced above we now provide an example for a setting in which 
$h_{GS}= h_{LS} < 1$.
Let the law $\textsf{P}^{m_0}$ of the mean offspring $m_0$ be given by
$$\frac{d\textsf{P}^{m_0}}{d\lambda}(x):=1.6\cdot \mathds{1}_{[0.5,1]}(x)+0.2\cdot \mathds{1}_{(1,2]}(x)$$
where $\lambda$ denotes the Lebesgue measure. Obviously $h_{LS}=0.5$ and a simple computation yields
$$\varphi(h_{LS})=0.2\cdot\Big(2\cdot\log(2)\Big)+1.6\cdot\Big(2\cdot\log(2)-1.5\cdot\log(3)\Big)<0.$$\\
\\
\textbf{Acknowledgement:}\\
We thank Sebastian M\"{u}ller for instructive discussions. We also thank the referees for useful comments which helped to improve parts of our results.

\newpage
\bibliographystyle{alpha}

\bigskip\bigskip\bigskip
\noindent
Christian Bartsch, Nina Gantert and Michael Kochler\\
CeNoS Center for Nonlinear Science and Institut f\"ur Mathematische Statistik\\
Universit\"at M\"unster\\
Einsteinstr. 62  \\
D-48149 M\"unster\\
Germany\\
{\tt christian.bartsch@uni-muenster.de}\\
{\tt gantert@uni-muenster.de}\\
{\tt michael.kochler@uni-muenster.de}\\

\end{document}